\documentclass[12pt]{amsart}
\usepackage{amsmath, amssymb, graphicx, hyperref}
\usepackage{epsfig}
\usepackage{autonum}
\usepackage{derivative}
\usepackage[usenames,dvipsnames,svgnames]{xcolor}
\usepackage[top=1in, bottom=1in, left=1in, right=1in]{geometry} 
\newcommand{\R}{\mathbb{R}}

\newcommand{\C}{\mathbb{C}}
\newcommand{\vphi}{\varphi}
\newcommand{\vt}{\vert}
\newcommand{\hf}[1][1]{\frac{#1}{2}}

\newcommand{\sumfty}[1][n]{\sum_{#1=1}^\infty}
\newcommand{\proj}{\operatorname{proj}}

\newcommand{\diff}{\mathop{}\!d}

\newcommand{\ta}{ \ \text{ and } \ }
\newtheorem{theorem}{Theorem}[section]
\newtheorem{conj}[theorem]{Conjecture}
\newtheorem{lemma}[theorem]{Lemma}
\theoremstyle{definition}

\theoremstyle{remark}
\newtheorem{remark}[theorem]{Remark}

\begin{document}
	\title{Non-zero values of a family of approximations of a class of $L$-functions}
	\author{Arindam Roy and Kevin You}
	\address{Department of Mathematics and Statistics, University of North Carolina at Charlotte, 9201 University City Blvd., Charlotte, NC 28223, USA}
	\email{aroy15@charlotte.edu}
	\address{Department of Mathematical Sciences
		Carnegie Mellon University
		Wean Hall 6113
		Pittsburgh, PA 15213}
	\email{kevinyou@andrew.cmu.edu}
	
	
	\begin{abstract}
		Consider the approximation $\tilde{Z}_N(s) = \sum_{n=1}^N n^{-s} + \chi(s) \sum_{n=1}^N n^{1-s}$ of the Riemann zeta function $\zeta(s)$, 
		where $\chi(s)$ is the ratio of the gamma functions. 
		This arise from the approximate functional equation of $\zeta(s)$. 
		Gonek and Montgomery have shown that $\tilde{Z}_N(s)$ has 100\% of its zeros lie on the critical line. 
		Recently, $a$-values of $\tilde{Z}_N(s)$ for non-zero complex number $a$ are studied and it has been shown that the $a$-values of $\tilde{Z}_N(s)$ are cluster arbitrarily close to the critical line. 
		In this paper, we show that, despite the above, 0\% of non-zero $a$-values of $\tilde{Z}_N(s)$ actually lie on the critical line itself. For $\zeta(s)$ at most $50\%$ non-zero $a$-values lie on the critical line is known due to Lester. We also extend our results to approximations of a wider class of $L$-functions.
	\end{abstract}
	
	\maketitle 
	
	\section{Introduction}
	Let us denote $s:=\sigma+it$ with $\sigma$ and  $t$ are real numbers. The Riemann zeta function $\zeta(s)$ is defined by the Dirichlet series $\zeta(s):=\sum_{n=1}^{\infty}n^{-s}$ for  $\sigma>1$.
	The functional equation of $\zeta(s)$ is given by
	\begin{align}
		\zeta(s) = \chi(s)\zeta(1-s)\quad \text{where}\quad 
		\chi(s) := \pi^{s-1/2}\frac{\Gamma((1-s)/2)}{\Gamma(s/2)}.
	\end{align}
	%This allows us to perform a meromorphic continuation to the whole complex plane except at $s=1$ where $\zeta(s)$ has a simple pole with residue equal to $1$. The connection with number theory comes from the Euler product
	%\[
	%\zeta(s) = \prod_{p} (1 - p^{-s})^{-1},
	%\]
	%for $\sigma>1$, and where the product is taken over all the primes $p$. 
	It is well-known from the works of Riemann and von Mangoldt that the non-trivial zeros of $\zeta(s)$ are located inside the critical strip $0 < \sigma <1$. Moreover, if $\beta+i\gamma$ denote non-trivial zeros of $\zeta(s)$ and $N(T)$ denotes the number of such zeros in the rectangle $0<\beta<t$ and $0 < \gamma <T$ then
	\begin{align}
		N(T) = \frac{T}{2 \pi} \bigg(\log \frac{T}{2\pi} -1 \bigg) + \frac{7}{8} +O(\log T)
	\end{align}
	%where
	%\begin{align}
	%S(T) = \frac{1}{\pi} \arg \zeta \bigg( \frac{1}{2} +it %\bigg) \ll \log T,
	%\end{align}
	as $T \to \infty$, see e.g. \cite{titchmarsh} for properties of $\zeta(s)$. Let $N_0(T)$ denote the number of non-trivial zeros with $\beta=1/2$ and  $0<\gamma<T$. We then define
	\begin{align}
		\kappa  = \mathop {\lim \inf }\limits_{T \to \infty } \frac{{{N_0}(T)}}{{N(T)}} .
	\end{align}
	The history behind the value of $\kappa$ can be found in \cite{feng}. Some significant milestones for the value of $\kappa$ as follows. In 1942, Selberg \cite{Selberg//posconst} establish that a non-zero proportion of zeros lie on the critical line, or $\kappa > 0$.
	Levinson \cite{Levinson//third} later showed in 1974 that $\kappa > .3474$. This was improved by Conrey \cite{Conrey} to $\kappa \ge .4088$ in 1989. The limit $\kappa = 1$ consider a weaker form of the Riemann hypothesis.
	
	Instead of asking when $\zeta(s) = 0$, one can ask for a fixed non-zero complex number $a \in \C$, when does $\zeta(s) = a$. 
	Levinson \cite{Levinson//a-vals} has showed that all such occurrences, 
	called $a$-values, lie arbitrarily close to the critical line. 
	However, it is conjectured by Selberg
	\cite{selberglost} that for $a \neq 0$, almost none of non-trivial $a$-values actually lie on the critical line, 
	unlike that for $a=0$ or Riemann's conjecture. 
	In 2015, Lester \cite{lester} proved that at most half of non-zero $a$-values of $\zeta(s)$ lie on the critical line. 
	A weaker version of Selberg's conjecture can be stated as
	\begin{conj}
		\label{conj:Lesterconjecture}
		Consider $a\in\mathbb C - \{0\}$. Then $0\%$ of the $a$-values of $\zeta(s)$ lie on the critical line in the sense
		\begin{align}
			\lim_{T\to\infty}\frac{N_{a,0}(T)}{N_{a}(T)} = 0
		\end{align}
		where $N_{a,0}(T)$ is the number of $a$-values on the critical line up to $T$, and $N_a(T)$ is the number of $a$-values up to $T$.
	\end{conj}
	Let $\Lambda(n)$ be the von-Mangoldt function and $\Lambda_X(n)=\Lambda(n)$ when $n\leq X$, $\Lambda_X(n)=\Lambda(n)(2-\log n/\log x)$ when $X<n\leq X^2$ and $\Lambda_X(n)=0$ otherwise. 
	Gonek \cite{gonek2007finite} introduced a function 
	\[ 
	P_X(s)=\exp\left(\sum_{n\leq X^2}\frac{\Lambda_X(n)}{n^s\log n}\right), 
	\]
	which approximates the partial Euler product $\prod_{p\leq X^2} (1 - p^{-s})^{-1}$ of $\zeta(s)$. Then he considered the function
	\begin{align}
		Z_N(s)=P_N(s)+\chi(s)P_N(\bar{s}), 
	\end{align}
	which is a good approximation of $\zeta(s)$, to demonstrate a model of $\zeta(s)$ exhibiting some basic features of it. If one assumes the Riemann hypothesis for $\zeta(s)$, Gonek proved that $Z_N(s)$ are good approximations of $\zeta(s)$ in $\sigma\geq1/2+\epsilon$, 
	$Z_N(s)$ and $\zeta(s)$ have approximately same numbers of zeros on the critical line,
	the zeros of $Z_N(s)$ are simple and in between two zeros of $\zeta(s)$, 
	and finally $Z_N(1/2+it)$ tends to $2\zeta(1/2+it)$ for large $N$. 
	Finally, he proved unconditionally the Riemann hypothesis for $Z_N(s)$ for $|t|>10$. One drawback on working with $Z_N(s)$ is the term $P_N(\bar{s})$. This made the functions $Z_N(s)$ non-analytic and many tools from complex analysis cannot be applied on $Z_N(s)$.
	
	Gonek and Montgomery \cite{Montgomery-Gonek} also considered the models $\Tilde{Z}_N(s)$ instead of $Z_N(s)$. Here 
	\begin{align}
		\Tilde{Z}_N(s):=\sum_{n\leq N}\frac{1}{n^s}+\chi(s)\sum_{n\leq N}\frac{1}{n^{1-s}}.   
	\end{align}
	Note that, a form of the approximate functional equation for the Riemann zeta function can be written as 
	\begin{align}
		\zeta(s)=\sum_{n\leq\sqrt{|t|/2\pi }}\frac{1}{n^s}+\chi(s)\sum_{n\leq \sqrt{|t|/2\pi}}\frac{1}{n^{1-s}}+O(|t|^{-\sigma/2})
	\end{align}
	for $|t|\geq 1$ and $|\sigma-1/2|\leq 1/2$. This shows that $\Tilde{Z}_N(s)$ can be a good approximation of $\zeta(s)$ for suitable choices of $N$. Spira \cite{Spira//approx_and_RH} first proved that $\Tilde{Z}_1(s)$ and $\Tilde{Z}_2(s)$ satisfies the Riemann hypothesis. 
	For $N \geq 3$, Gonek and Montgomery \cite{Montgomery-Gonek} 
	showed that almost all ($100\%$) non-trivial zeros of $\Tilde{Z}_N(s)$ lie on the critical line, in the limit $T \rightarrow \infty$. Distributions of zeros of approximations of other $L$-functions are also studied in \cite{LiNataRoy,LiRoyZah}. For any non-zero complex number $a$, Crim, Frendreiss and the first author initiated the study of the $a$-values of $\Tilde{Z}_N(s)$ in \cite{CrFrRo}. They showed in \cite{CrFrRo} that the $a$-values of  $\Tilde{Z}_N(s)$ are dense to the critical line similar to the denseness of $a$-values of $\zeta(s)$ (see \cite{Levinson//a-vals}). Although the $a$-values of $\zeta(s)$ on the critical line are studied and it is  conjectured that $0\%$ of total non-zero $a$-values lie on the critical line due to Selberg, nothing are known about the $a$-values of $\Tilde{Z}_N(s)$ on the critical line when $a\neq 0$. 
	
	Our main goal in this paper to show that the weaker version of Selberg's Conjecture \ref{conj:Lesterconjecture} holds true for $\Tilde{Z}_N(s)$. In particular we have the following result. 
	\begin{theorem}\label{firstmainthm}
		Let $a\neq 0$. Then $0\%$ of $a$-values of $\Tilde{Z}_N(s)$ lie on the critical line. 
	\end{theorem} 
	
	Actually, we establish the above result in more general setting, i.e.,  to the approximations like $\tilde{Z}_N$ to a wider class of $L$-functions. 
	A general Dirichlet series is defined by 
	\begin{align} 
		F(s) = \sumfty \frac{a_n}{\lambda_n^s}, \label{def:F}
	\end{align}
	where $a_n$ are complex numbers, $\lambda_n$ are real numbers with $\lambda_n \geq 1$, and $\lambda_n$ strictly increasing and the series is absolutely convergent in some half-plane. 
	We let the first term be the constant term, or $\lambda_1 = 1$. 
	It is possible that $a_1 = 0$. For our results we assume that $F$ has at least two non-constant terms, or $a_2\neq 0$ and $a_3\neq 0$. 
	We also assume that $\vt a_n \vt$ and $\lambda_n$ are both grow like a polynomial in $n$.
	
	Let us assume that $F$ satisfies a functional equation of the form
	\begin{align} \Omega(s) F(s) = \lambda^{2s - \delta} \Omega(\delta - s) F(\delta - s) \end{align} 
	for some $\lambda > 0$, 
	$\delta \in \R$ and $\Omega(s)$ is a product of Gamma functions
	\begin{align} \Omega(s) = \prod_{i=1}^k \Gamma(\alpha_i s + \beta_i) \end{align}
	with $\alpha_i > 0$ and $\beta_i \in \R$.
	Chandrasekharan and Narasimhan \cite{CHANDRASEKHARAN1963} proved that, under suitable conditions, $
	F$ has an approximate functional equation of the form
	\begin{align} F(s) = \sum_{\lambda_n \leq x} \frac{a_n}{\lambda_n^s} + \lambda^{2s-\delta} \frac{\Omega(\delta-s)}{\Omega(s)} 
		\sum_{\lambda_n \leq y} \frac{a_n}{\lambda_n^{\delta-s}} + R(x,y,\vt t \vt). \end{align} 
	Given $F(s)$, $\Omega(s)$, and $\lambda > 0$ of the form above, we will be investigating the approximations
	\begin{align} 
		\zeta_N(s) = F_N(s) + G(s) F_N(\delta-s),\label{def:zetN}
	\end{align}
	where we have defined
	\begin{align} 
		F_N(s) := \sum_{n=1}^N \frac{a_n}{\lambda_n^s} \ta G(s) := \lambda^{2s-\delta} \frac{\Omega(\delta-s)}{\Omega(s)}. \label{def:Del}
	\end{align}
	%This generalizes the function $\Tilde{Z}_N(s)$ considered by Spira and Gonek and Montgomery.
	Note that $G(s) G(\delta-s) = 1$ and $\zeta_N(s) = G(s) \zeta_N(\delta-s)$. %This always holds by construction and does not rely on the functional equation between $F(s)$ and $G(s)$.% \section{Results}
	We use the convention that the critical line is the line of symmetry $\sigma = \hf[\delta]$. We also make a note that for a given $a \in \C$, an $a$-value of $f(s)$ is a zero of $f(s) - a$,
	counted with multiplicity. 
	
	Let us denote the number of $a$-values of $\zeta_N(s)$ within the region $T < t < T + U$ by $N_a(T, U)$.  Additionally, the $a$-values lying within $\vt \sigma - \hf[\delta] \vt < \epsilon$ of the critical line is denoted by $N_{a,\epsilon}(T,U)$. Also, the $a$-values lying on the critical line $\sigma = \hf[\delta]$ is denoted by $N_{a,0}(T,U)$.
	%\begin{itemize}
	%\item[a)] within the region $T < t < T + U$ be $N_a(T, %U)$.
	%\item[b)] additionally lying within $\vt \sigma - %\hf[\delta] \vt < \epsilon$ of the critical line be %$N_{a,\epsilon}(T,U)$. 
	%\item[c)] additionally lying on the critical line %$\sigma = \hf[\delta]$ be $N_{a,0}(T,U)$. 
	% \end{itemize}
	Through out the paper, we always assume that $T,U$, and $\epsilon$ are chosen so that there are no $a$-values 
	lying on the boundary of the rectangle require for the proof or endpoints of the line segment. This can be achieved since the $a$-values are finite in a bounded region. 
	We are interested in $a$-values of $\zeta_N(s)$ for sufficiently large $t$, as there may be infinitely many
	regularly distributed $a$-values of $\zeta_N(s)$ for small $t$. 
	In the first theorem we give an asymptotic of the count of the $a$-values in a rectangle. We make explicit the dependence of the error terms on $N$.  For convenience, let   
	$ A = \sum_{i=1}^k \alpha_i$ and $ B = \sum_{i=1}^k \alpha_i \log \alpha_i - \alpha_i$. 
	\begin{theorem}\label{Count}
	For any $a \in \C$, 
	there exists some $T_a, \gamma > 0$ such that if $T \geq T_a N^\gamma$, then 
	\begin{align} N_a(T,U) &=  \frac{A}{\pi}  \left( (T+U) \log (T+U) - T \log T \right) + \frac{B - \log \lambda + \Psi}{\pi} U+ O_a(N^\gamma \log (T+U)) \end{align} 
	where 
	\begin{align}
		2\Psi = \begin{cases}
			\log \lambda_2 &\text{ if } a \neq a_1 = 0 \\
			-\log \lambda_2 &\text{ if } a = a_1 \neq 0 \\
			0 &\text{ otherwise}
		\end{cases}. 
	\end{align}
	\end{theorem}
	\begin{remark} Our results are generalized the results of \cite{Montgomery-Gonek}, 
	and \cite{CrFrRo}. For example if we choose $A = \hf, B = \hf \log \hf - \hf$, $\lambda = \sqrt{\pi}$, 
	and $a_1 = 1$, so the expression
	for $N_a(T,U)$ reduces to 
	\begin{align} 
		N_a(T,U) &=  \frac{1}{2\pi}  \left( (T+U) \log (T+U) - T \log T \right) + \frac{U}{2\pi} \log \frac{1}{2\pi} + 
		\frac{2 \Psi - 1}{2\pi} U \\
		&\hspace{3in}+ O_a(N^\gamma \log (T+U)) \end{align}
	which agrees with that obtained in  \cite{CrFrRo}. 
	We additionally specified that 
	the error term is bounded by a polynomial in $N$. 
	This theorem counts the total number of $a$-values with imaginary part bounded.
	%Fixing $T$, $N_a(T,U)$ grows like $\Theta(U \log U)$, like that of the Riemann zeta function. 
	\end{remark}
	In the next theorem we give an asymptotic of the counts the $a$-values those are at least $\epsilon$-distance from  the critical line. 
	\begin{theorem}\label{Cluster}
	For any $a \in \C$, there exists some $T_a, \gamma > 0$ such that if $T \geq T_a N^\gamma$, then 
	for any $\epsilon > 0$, 
	\begin{align} N_a(T,U) - N_{a,\epsilon}(T,U) = O_a(U \log N /\epsilon) + O_a(N^{2\gamma} \log(T+U)/\epsilon). \end{align}
	\end{theorem}
	%This results shows that the is much smaller than the total number of $a$-values.
	Along
	with Theorem \ref{Count}, we can conclude that for any fixed $\epsilon$, 
	almost all $a$-values appear within $\epsilon$-distance of the critical line, 
	or in other words, $a$-values are clustered about the critical line. This generalizes the work in \cite{CrFrRo} on the cluster-ness of $a$-values of  $\tilde{Z}_N(s)$.
	
	In the next two theorems we pay our attention of all $a$-values on the critical line. First we give a lower bound of the number of zeros of $\zeta_N(s)$ on the critical line. For $a\neq 0$, we give an upper bound of the number of $a$-values of $\zeta_N(s)$ on the critical line. 
	\begin{theorem}\label{Critical Zero}
	Suppose that the coefficients $a_n$ are real. Then there exists some $T_0 > 0$ such that if $T \geq T_0$, then
	\begin{align} N_{0,0}(T,U) \geq \frac{A}{\pi}   \left( (T+U) \log (T+U) - T \log T \right) + O(U \log N). \end{align}
	\end{theorem}
	%This result lower bounds the number of zeros on the critical line. 
	In conjunction with Theorem \ref{Count}, by taking $N\leq T^{o(1)}$ and $U\geq T^b$  for any fixed positive real number $b$, Theorem \ref{Critical Zero} gives
	\begin{align}
	\liminf_{T\to \infty}\frac{N_{0,0}(T,U)}{N_{0}(T,U)}=1.
	\end{align}
	Hence, this theorem suggests that $100\%$ zeroes of $\zeta_N(s)$ lie on the critical line,  which generalizes the result for $\tilde{Z}_N(s)$ given in \cite{Montgomery-Gonek} to an approximation of general $L$-functions. 
	%Our final theorem places an  upper bound the number of nonzero $a$-values of $\zeta_N(s)$ on the critical line.This result 
	
	Our next result says that our original result, Theorem \ref{firstmainthm}, also works for a wider class of L-functions. 
	\begin{theorem}\label{Critical}
	Suppose that the coefficients $a_n$ are real. Then there exists some $T_0, U_0, \nu > 0$ such that for any non-zero $a \in \C $ and $t \geq T_0$,
	all $a$-values of $\zeta_N(s)$ on the critical line are simple. Moreover,
	for $U \geq U_0 N^\nu$, 
	\begin{align} N_{a,0}(T,U) = O(U \log N) \end{align}
	where the constant is independent of $a$. 
	\end{theorem}
	
	%This is the last piece of the puzzle, which upper bounds the $a$-values on the critical line.
	As an immediate consequence of Theorem \ref{Critical} together with Theorem \ref{Count}, 
	\begin{align}
	\lim_{T\to \infty}\frac{N_{a,0}(T,U)}{N_{a}(T,U)}=0
	\end{align}
	as long as $N\leq T^{f(T)}$ for some $f(T)=o(1)$ and $U\geq T^b$ for some fixed positive real number $b$. In another words we can say $0\%$ of non-zero $a$-values of $\zeta_N(s)$ lie on the critical line, which proves the Conjecture \ref{conj:Lesterconjecture} for $\zeta_N(s)$. As a special case if one takes $\zeta_N = \Tilde{Z}_N$ then we have Theorem \ref{firstmainthm}. %This suggest that in any bounded vertical part of critical line, the proportion of $a$-values tends to zero as $T\to \infty$. 
	Along with Theorem \ref{Count} and \ref{Cluster}, it suggests that, even though non-zero $a$-values of $\zeta_N(s)$ lie very close to the critical line,
	almost all of them do not lie on the critical line, contrary to the case $a=0$.
	
	Our last result is to give a region where all $a$-values lie. This result is essential for understanding the previous theorems.
	\begin{theorem}\label{Strip}
	There exists $T_0 > 0$ such that for any $a \in \C$, there exist positive numbers $\sigma_a$ and $\gamma$ such that all $a$-values of $\zeta_N(s)$ with $t \geq T_0$ lie within the region $\vt \sigma \vt < \sigma_a N^\gamma$. Moreover, if $\sigma \geq \sigma_a N^\gamma$ and $t \geq T_0$, then $\zeta_N(s) - a \in D(\vt a_1-a \vt,\vt a_1-a \vt/2)$ when $a \neq a_1$ otherwise, 
	$(\zeta_N(s)-a) \lambda_2^s \in D(a_2,\vt a_2 \vt/2)$. If $\sigma \leq - \sigma_a N^\gamma$ and $t \geq T_0$, then 
	$\vt \zeta_N(s) - a \vt > 1$.
	\end{theorem}
	
	Theorem \ref{Strip} asserts that for sufficiently large imaginary part, all $a$-values of $\zeta_N(s)$
	lie in some critical strip, and this critical strip is bounded by a polynomial in $N$. 
	Moreover, the theorem places some constrains of the behavior of $\zeta_N(s)$ outside the critical strip. 
	On the right, $\zeta_N(s)$ is near $a_1$. 
	On the left, $\zeta_N(s)$ is far from $a$. 
	
	\begin{remark}
	Theorem \ref{Count}, \ref{Cluster} and \ref{Strip}, it
	is possible to take any $\gamma$ such that $2A \gamma > \mu$ if $\mu$ 
	satisfies $\vt F_N(s) \vt \ll N^{\mu \vt \sigma \vt + \nu}. $ For example, if $F$ is the
	Riemann Zeta function, $\mu = 1, A = \hf$, and any $\gamma > 1$ works. 
	\end{remark}
	\begin{remark}
	For Theorem \ref{Critical}, any $\nu$ such that
	$\sum_{n=1}^N a_n \lambda_n^{-\hf[\delta]} \log \lambda_n \ll N^\nu$ works. 
	For example, if $F$ is the
	Riemann Zeta function, any $\nu > \hf$ works. 
	\end{remark}
	\section{Preliminaries Lemmas}
	
	We first reference the Littlewood lemma \cite[p.~10]{Gonek_notes}, bearing resemblance to the argument principle, that counts the expected real part of zeroes of a holomorphic function within a rectangle.
	\begin{lemma}
	Let $f(s)$ be a holomorphic functions in and upon the boundary of a rectangle $R=[\sigma_l,\sigma_r]\times [0,T]$, and non-zero on the three
	edges $t = 0, T$ and $\sigma = \sigma_r$. Then, by defining $\arg(f(s))$ continuously
	along the three edges, 
	\begin{align} 2
		\pi \sum_{\rho \in R} \left(\Re(\rho) - \sigma_l\right) &= \int_0^T \log \vt f(\sigma_l + it) \vt 
		- \log \vt f(\sigma_r+it) \vt \diff t \\
		&\quad+ 
		\int_{\sigma_l}^{\sigma_r} \arg (f(\sigma+iT)) - \arg(f(\sigma)) \diff \sigma 
	\end{align} 
	where the sum is taken over all zeros $\rho$ of $f$ in $R$.
	\end{lemma}
	
	Next, we recall the following form of Stirling's approximation of the Gamma function from \cite[\S 6.1]{abramowitzstegun}, which is essential
	in characterizing the behavior of the $G(s)$ term defined in \eqref{def:Del}.  
	
	\begin{lemma}\label{Stirling}
	Fix any $\epsilon > 0$, for all $s$ with $\vt \arg (s) \vt < \pi - \epsilon$ and $s \to \infty$,
	\begin{align} \Gamma(s) = \sqrt{\frac{2\pi}{s}} \left( \frac{s}{e} \right)^s \left( 1 + O(\vt s \vt^{-1})\right).  \end{align} 
	\end{lemma}
	We also need the following form of asymptotic of Digamma function \cite[\S 6.3]{abramowitzstegun}.
	\begin{lemma}\label{digamma}
	Fix any $\epsilon > 0$,  for all $s$ with $\vt \arg (s) \vt < \pi - \epsilon$ and $s \to \infty$,
	\begin{align}
		\psi(s)=\frac{\Gamma'(s)}{\Gamma(s)}\sim \ln s-\frac{1}{2s}-\frac{1}{12s^2}+\cdots.
	\end{align}
	\end{lemma}
	\begin{lemma}\label{Chi Bound}
	For sufficiently large $\sigma$ and sufficiently large $t$,
	\begin{align} \left\vt \frac{1}{G(s)} \right\vt \gg \lambda^{2\sigma-\delta} 
		\prod_{i=1}^k e^{\alpha_i t} \left( \frac{\vt \alpha_i s \vt}{2e} \right)^{ 
			\alpha_i (2\sigma - \delta)}.   \end{align} 
			\end{lemma}
			
			\begin{proof}
	From \eqref{def:Del} it is clear that $G(s)$ is a product of ratios of Gamma functions. Consider a single term
	of $1/G(s)$ with
	\begin{align}
		\frac{\Gamma(\alpha s + \beta)}{\Gamma(\alpha (\delta-s) + \beta)} 
		%&= \frac{\Gamma(\alpha s + \beta)}{\Gamma(1 + \alpha \delta - 1 - \alpha s + \beta)} \\
		&= \Gamma(\alpha s+\beta) \Gamma(\alpha s + 1  - \alpha \delta - \beta) \frac{\sin(\pi(\alpha s + 1 - \alpha \delta - \beta))}
		{\pi},  
	\end{align}
	where we have used the reflection formula for the Gamma function
	\begin{align}
		\Gamma(z)\Gamma(1-z)=\frac{\pi}{\sin(\pi z)}
	\end{align}
	for $z\notin \mathbb{Z}$. Now for any $x,y\in \mathbb{R}$ one has
	\begin{align} 
		\vt \sin(x+iy) \vt^2 = \sin^2 x \cosh^2 y + \cos^2 x \sinh^2 y \geq \sinh^2 y . 
	\end{align}
	Hence
	%\textcolor{red}{
		\begin{align} 
			\vt \sin (\pi(\alpha s + 1 - \alpha \delta - \beta)) \vt \geq \sinh (\pi\alpha t) = \frac{e^{\pi\alpha t} - e^{-\pi \alpha t}}{2} \geq
			\frac{e^{\pi\alpha t}}{4} \geq \frac{e^{\alpha t}}{4}
		\end{align}
		%}
	when $2 \pi\alpha t \geq \log 2$. We take $t$ sufficiently large so that this holds.
	Additionally, for $\sigma$ large enough 
	%so that 
	%$\alpha s + \beta,\alpha s+1-\alpha \delta - %\beta$ have 
	%large positive real parts, larger than $\alpha s/2$, and with the error term 
	%of Stirling's formula \ref{Stirling} is bounded by $1/2$.
	and by Stirling's formula \eqref{Stirling} we obtain
	\begin{align}
		\frac{\Gamma(\alpha s + \beta)}{\Gamma(\alpha (\delta-s) + \beta)} 
		\gg& \ e^{\alpha t} \left( \frac{\vt \alpha s \vt}{2e} \right)^{\alpha \sigma + \beta - \hf}
		\left( \frac{\vt \alpha s \vt}{2e} \right)^{\alpha \sigma + 1 - \alpha \delta - \beta - \hf} \\
		=& \ e^{\alpha t} \left( \frac{\vt \alpha s \vt}{2e} \right)^{\alpha (2\sigma - \delta)}.
	\end{align}
	Taking the product over all terms of $1/G(s)$ yields the desired result of the lemma. 
	\end{proof} 
	As a consequence of Lemma \ref{Chi Bound} we have following bound for $G(s)$.
	\begin{lemma}\label{Inequality}
	For any $\epsilon > 0$, $\mu,\nu \in \R$, $c \geq 1$, 
	and sufficiently large $t$, there exists $\sigma_a,\gamma > 0$ such that $\sigma \geq \sigma_a N^{\gamma}$ implies 
	\begin{align} \vt G(s) \vt N^{\mu \sigma + \nu} c^{\sigma} < \epsilon. \end{align}
	\end{lemma}
	
	\begin{proof}
	Take $\sigma$ and $t$ large enough to satisfy lemma \ref{Chi Bound}. We need
	\begin{align} \lambda^{2\sigma} c^{-\sigma} \prod_{i=1}^k \left( \frac{\alpha_i \sigma}{2e} \right)^{\alpha_i (2\sigma - \delta)} 
		\gg N^{\mu \sigma + \nu}. \end{align}
	Taking the logarithm of both sides, we need 
	\begin{align} 2 \sigma \log \lambda - \sigma \log c + \sum_{i=1}^k \alpha_i (2\sigma - \delta) \log \frac{\alpha_i \sigma}{2e}
		\geq (\mu \sigma + \nu ) \log N + C \end{align}
	for some $C$. Let $\sigma \geq \sigma_a N^\gamma$ and $A = \sum_{i=1}^k \alpha_i$. It 
	suffices to have
	\begin{align} A (2\sigma-\delta) \log N^\gamma \geq (\mu \sigma + \nu) \log N
	\end{align} and  \
	\begin{align} 2 \sigma \log \lambda - \sigma \log c + A(2\sigma-\delta)
		\log \frac{\max \alpha_i \sigma_a}{2e} \geq C. \end{align}
	This is possible by having $2A \gamma > \mu$ and letting $\sigma_a$ be large enough. 
	\end{proof}
	
	The following bounds of $F_N(s)$ are necessary for the proof of theorems. 
	\begin{lemma}\label{Fn bound}
	Let $F_N(s)$ be defined in \eqref{def:Del}. Then there exist positive constant $\mu$ and $\nu$ such that
	\begin{align} \vt F_N(\delta-s) \vt \ll N^{\mu \vt \sigma \vt + \nu}\label{fbdd}    \end{align}
	holds for all $\sigma$ and 
	\begin{align}  
		\vt F_N(s) - a_1 \vt \ll \lambda_2^{-\sigma}\label{a1bdd}  
	\end{align}
	for sufficiently large $\sigma$ independent of the choice of $N$. 
	\end{lemma}
	
	\begin{proof}
	The first inequality holds trivially since $a_n$ and $\lambda_n$ are bounded by a polynomial in $N$.  
	The second inequality holds for all $\sigma \geq \sigma_0$ if $\sigma_0$ is in the abscissa of convergence of $F$,
	since $\vt F_N(s) - a_1 \vt \leq \vt F_N(\sigma) - a_1 \vt \leq \lambda_2^{-\sigma + \sigma_0} F_N(\sigma_0)$. 
	\end{proof}
	
	%  Next, we combine the estimates of $F_N(s)$ and $G(s)$ a single inequality. 
	%This is possible since $\chi(\delta-s)$
	% grows super-exponentially in $\sigma$, while $F_N(\delta-s)$ only grows exponentially. 
	The following lemma will be essential in later part of this article. 
	
	\begin{lemma}\label{projlemma}
	Let $\alpha \in [0,2\pi)$ and $z \in \C$. Then, there exists unique $x,y \in \R$ such that 
	$z = xe^{i\alpha} + ye^{i(\alpha+\pi/2)}$. Moreover, if we define 
	\begin{align}
		\proj_\alpha z := x,
	\end{align} then $\proj_\alpha$
	is linear and 
	\begin{align} 2 \proj_\alpha z = z e^{-i\alpha} + \overline{z} e^{i\alpha} = 2 \Re(z e^{-i\alpha}) \end{align}
	\end{lemma}
	
	\begin{proof}
	Immediate.
	%Since $e^{i\alpha}$ and $e^{i\alpha+\pi/2}$ form a basis of $\C$ over $\R$ then $z = xe^{i\alpha} + ye^{i\alpha+\pi/2}$ and $\proj_\alpha$
	%is linear follow immediately. Finally, we can verify that 
	%\begin{align} z =  \left( \proj_\alpha z \right) e^{i\alpha} + \left( \proj_{\alpha+\pi/2} z \right) e^{i\alpha+\pi/2} . \end{align}
	%By a simple computation we find $2 \proj_\alpha z = z e^{-i\alpha} + \overline{z} e^{i\alpha}$. 
	\end{proof}
	
	Another lemma is necessary to obtain behavior of trigonometric functions.
	
	\begin{lemma}\label{Trig}
	Let $\omega_1,\omega_2 > 0$, $\phi_1,\phi_2 \in \R$ be with $\omega_1 \neq \omega_2$. Then, for any 
	interval $I$, 
	\begin{align} \left\vt \int_I \sin(\omega_1 t + \phi_1) \sin(\omega_2 t + \phi_t) \diff t \right\vt \ll \frac{1}{\vt \omega_1 - \omega_2 \vt}. \end{align}
	\end{lemma}
	
	\begin{proof}
	The result follows immediately by writing the integrand as a sum of cosine functions ans using the fact $\vt \omega_1 - \omega_2 \vt\leq \vt \omega_1 +\omega_2 \vt$.
	\end{proof}
	%Following simple lemma is the heart of the theorems \ref{Count} and \ref{Cluster}. It allows us to reflect $\zeta_N(s) - a$
	% for negative $\sigma$, when $\zeta_N(s)$ is ill-behaved, to $\zeta_N^*(s)+a$ for positive $\sigma$,
	% where $\zeta_N^*(s)$ is nicely contained within a disk.
	
	%\begin{lemma}\label{Reflection}
	% Define $F_N^*(s) = F_N(s) - a$, and $\zeta_N^*(s) = F_N^*(s) + %G(s) F^*_N(\delta-s)$. Then,
	%\begin{align} \zeta_N(s) - a = %G(s) (\zeta_N^*(\delta-%s)+a). \end{align}
	%\end{lemma}
	%\begin{proof} Recall the fact $G(s) \Delta(\delta-s) = 1$. Then
	%   \begin{align} G(s) (\zeta^*_N(\delta-s) + a) &= G(s) (F_N(\delta-s) + \chi(\delta-s) (F_N(s)-a))\\
	%  & = G(s)F_N(\delta-s)+F_N(s)-G(s) \Delta(\delta-s) a\\
	%   &=\zeta_N(s) - a. \end{align}
	% This completes the proof the lemma.
	% \end{proof}

\begin{lemma}\label{Monotone}
For any fixed $\sigma$ and sufficiently large $t$, $ \pdv{}{t} \arg G(\sigma+it) < 0$. 
\end{lemma}

\begin{proof}
The proof of this lemma relies on the asymptotics of the digamma function $\psi(s)$. 
As usual, we consider each term of $\chi(s)$ separately.
We have that 
\begin{align} \pdv{}{t} \arg \Gamma(\alpha s + \beta) = \Im \left( \pdv{}{t} \log \Gamma(\alpha s + \beta) \right) = 
\alpha \Re \left( \psi (\alpha s + \beta) \right). \end{align}
By Lemma \ref{digamma}, $\Re \left( \psi (\alpha s + \beta) \right)$ tends to infinity  for fixed $\sigma$ and $t\to\infty$. 
So, $\pdv{t} \arg \Gamma( \alpha s + \beta)$
can be made arbitrarily large with large $t$. Consequently, 
\begin{align} \pdv{}{t} \arg(G(\sigma+it)) =& \pdv{}{t} \sum_{i=1}^k \arg (\Gamma(\alpha_i s + \beta)) 
- \arg(\Gamma(\alpha_i s + \beta_i)) \\
&+ \log \lambda < 0\end{align}
for sufficiently large $t$. 
\end{proof}

%\section{Proofs}

\section{Proof of Theorem \ref{Strip}}

We prove the theorem by showing that there exist positive number $T_0, \gamma$ and $\sigma$ so that there are no zeros of $\zeta_N(s)-a$ outside the region $\{\sigma+it:|\sigma|\leq \sigma_aN^{\gamma}, t>T_0\}$.
We split the proof in two cases whether $a \neq a_1$ and $a = a_1$.
%We fix a positive number $T_0$ that satisfy both the hypothesis of combined inequality \ref{Inequality}. 
%We find $\sigma_a,\gamma > 0$ for the left and right separately,
%and then take the larger of the two.
%In this section we derive a zero free region for $\zeta_N(s)-a$. 

In the first case, since $a\neq a_1$ and $\lambda_2>1$ by the definition, then by \eqref{a1bdd} of Lemma \ref{Fn bound}, we find that
\begin{align} 
\vt F_N(s) - a_1 \vt\ll \frac{1}{\lambda_2^{\sigma}}< \frac{\vt a_1 - a \vt}{4}\label{ineq01}
\end{align}
for suitable choices of $\sigma$. Now, we choose $T_0$, $\sigma_a$ and $\gamma$ such that 
for all $\sigma \geq \sigma_a N^\gamma$, the inequality \eqref{ineq01} holds as well as the bounds in Lemma \ref{Inequality} and Lemma \ref{Fn bound} can be obtained. Hence for all $\sigma \geq \sigma_a N^\gamma$ we have
\begin{align}
\vt G(s) F_N(\delta-s)  \vt \leq \vt G(s) \vt N^{\mu \sigma + \nu} < \frac{\vt a_1 - a \vt}{4}. \label{ineq02}
\end{align}
Therefore, from the definition \eqref{def:zetN} of $\zeta_N(s)$ and from the \eqref{ineq01} and \eqref{ineq02} we find 
\begin{align}
\vt \zeta_N(s) - a_1 \vt \leq \hf\vt a_1 - a\vt.         \end{align}
If instead $a = a_1$, then we apply Lemma \ref{Fn bound} and Lemma \ref{Inequality} to the truncated
series $(F_N(s)-a_1) \lambda_2^{s}$. Hence for suitable choices of $\lambda_a$ and $\gamma$ and for all $\sigma\geq \sigma_a N^{\gamma}$ we have 
\begin{align} 
\vt (F_N(s) - a_1) \lambda_2^{s} - a_2 \vt < \frac{\vt a_2 \vt}{4}  
\end{align} 
and
\begin{align} \vt G(s) F_N(\delta-s) \lambda_2^s \vt < \frac{\vt a_2 \vt}{4}. \end{align}
Hence, by combining above two inequalities, we have
\begin{align}\vt (\zeta_N(s) - a)\lambda_2^s - a_2 \vt \leq \hf[\vt a_2 \vt].
\end{align}
This proof the first part of Theorem \ref{Strip}.

To prove the next part first we show that $\vt \zeta_N(s) \vt > \vt a \vt + 1$ as $\sigma\to -\infty$. 
As previous we consider  two cases, with $a_1 \neq 0$ and $a_1 = 0$.
Hence from \eqref{def:zetN} and the triangle inequality, it suffices to show that if $\sigma \leq -\sigma_aN^{\gamma}$, then
\begin{align} \left\vt G(s) F_N(\delta-s) \right\vt  \frac{\vt F_N(\delta-s) \vt}{\vt G(\delta-s) \vt} > \vt F_N(s) \vt + \vt a \vt + 1. \end{align}
In other words, we need to show that 
\begin{align}
\frac{\vt F_N(\delta-s) \vt}{\vt G(\delta-s) \vt} > \vt F_N(s) \vt + \vt a \vt + 1
\end{align}
or
\begin{align} \vt G(\delta-s) \vt (\vt F_N(s) \vt + \vt a \vt + 1) < \vt F_N(\delta-s) \vt. \end{align}
In the former case, let $\sigma$ be sufficiently negative so that, 
in view of \eqref{a1bdd} of Lemma \ref{Fn bound}, $\vt F_N(\delta-s) \vt > \hf[\vt a_1 \vt]$.
Then, by Lemma \ref{Inequality} find $T_0$,$\sigma_a, \gamma$ such that 
\begin{align} \vt G(\delta-s) \vt (N^{\mu \sigma + \nu} + \vt a \vt + 1) < \hf[\vt a_1 \vt]. \end{align}
for $\sigma \leq -\sigma_a N^\gamma$.
If instead $a_1 = 0$, we use $\vt F_N(\delta-s) \vt > \hf[\lambda_2^{\sigma}]$, 
and find $T_0$, $\sigma_a, \gamma$ for which
\begin{align} \vt G(\delta-s) \vt (N^{\mu \sigma + \nu} + \vt a \vt + 1) < \hf[\lambda_2^{\sigma}]. \end{align}
This concludes the proof.

\section{Proof of Theorem \ref{Count}} \label{pthmzerocount}

Let use define a new general Dirichlet series with $F_N^*(s) := F_N(s) - a$ and also deonte $\zeta_N^*(s) := F^*(s) + G(s) F^*(\delta-s)$. 
Using Theorem \ref{Strip}, 
we know that all $a$-values of $\zeta_N(s)$ with $t>T_1$ lie within $\vt \sigma \vt \leq \sigma'_a N^{\gamma'}$
for some $\sigma'_a,\gamma' > 0$. Likewise, 
we have that all $-a$-values of $\zeta_N^*(s)$ with $t>T_2$ lie
within $\vt \sigma \vt \leq \sigma_a^* N^{\gamma^*}$ for some
$\sigma_a^*,\gamma^* > 0$.
We choose $T_a$ be the maximum of $T_1$ and $T_2$, $\sigma_a$ be the maximum of $\sigma'_a$ and $\sigma_a^*$ and $\gamma$ be the maximum of $\gamma'$ and $\gamma^*$ for this proof. We also assume that $T_a \geq 4 \sigma_a$, 
that  $t \geq T_a$ with $\epsilon = \pi/4$ gives $\vt O(\vt s^{-1} \vt)\vt < 1$ in Lemma \ref{Stirling},
and finally that Lemma  \ref{Chi Bound} holds for $t \geq T_a N^\gamma - 4r$.

Let us define the function
\begin{align} h_N(s) := \zeta_N(s) - a = F_N(s) + G(s) F_N(\delta-s) - a.\label{hNdef} \end{align}
Consider the rectangular contour $\phi$ oriented 
counterclockwise with vertices $r+iT, r+i(T+U), -r+i(T+U)$, and $-r+iT$, where $r = \sigma_a N^\gamma$. Then by the argument principle,
\begin{align} N_a(T, U) = \frac{1}{2\pi} \Delta_\vphi \arg(h_N(s)).\label{arghn} \end{align}
We compute the argument for each of the sides $\vphi_1,\vphi_2,\vphi_3$,and $\vphi_4$ of the rectangle $\phi$, where $\vphi_1$ is the right vertical edge, $\vphi_2$ is the top edge, $\vphi_3$ is the left vertical edge and $\vphi_4$ is the bottom edge.

\subsection{On the right vertical edge}
By Theorem \ref{Strip}, we are guaranteed that if $a \neq a_1$, then $\Delta_{\vphi_1} \arg (h_N(s)) \leq \pi$,
since the image of $h_N(s)$ stays within a disk not containing the origin.
If instead $a = a_1$, then
\begin{align} \vt \Delta_{\vphi_1} \arg(h_N(s)) + \Delta_{\vphi_1} \arg(\lambda_2^s) \vt = \vt \Delta_{\vphi_1} \arg(h_N(s) \lambda_2^s) \vt \leq \pi. \end{align}
Therefore, 
\begin{align}
\Delta_{\vphi_1} \arg(h_N(s)) = - U \log \lambda_2 + O(1).\label{bdphi1}
\end{align} 

\subsection{On the left vertical edge}
Recall the fact $G(s) \Delta(\delta-s) = 1$. Then
\begin{align} 
h_N(s) &= G(s) (F_N(\delta-s) + \Delta(\delta-s) (F_N(s)-a))\\
&=G(s) (\zeta^*_N(\delta-s) + a).\label{Reflection}
%& = G(s)F_N(\delta-s)+F_N(s)-G(s) %\Delta(\delta-s) a\\
%&=\zeta_N(s) - a. 
\end{align}
Note that $\Delta_{\vphi_3}\arg h_N(s)= \Delta_{\vphi_3}\arg G(s)+\Delta_{\vphi_3}\arg(\zeta_N^*(\delta-s) + a )$. To compute $\Delta_{\vphi_3}\arg(\zeta_N^*(\delta-s) + a )$,
%  Recall that by the reflection formula \ref{Reflection}, 
% \begin{align} h_N(s) = \zeta_N(s) - a = \chi(s) (\zeta_N^*(\delta-s)+a). %\end{align}
we can imitate the previous computation on the right edge and apply Theorem \ref{Strip} to $\zeta_N^*(s)$ for $-a$-value. Then
we have 
\begin{align}\Delta \arg_{\vphi_3} (\zeta_N^*(\delta-s) + a )
= \begin{cases} 
	O(1) &\text{ if } a_1 \neq 0 \\
	U \log \lambda_2 + O(1) &\text{ if } a_1 = 0
	\end{cases} \end{align}  
	Now we compute $\Delta_{\vphi_3} \arg G(s)$.
	% For a single $\Gamma(\alpha s + \beta)$ term, 
	From Lemma \ref{Stirling} one has
	\begin{align}
& \Delta_{\vphi_3} \arg(\Gamma(\alpha s + \beta)) \\ = \ & 
- \Delta_{\vphi_3} \arg\left( (\alpha s + \beta)^{\alpha s + \beta} e^{-\alpha s - \beta} \sqrt{\frac{2\pi}{\alpha s + \beta}} (1 + O(\vt (\alpha s + \beta)^{-1} \vt)) \right).
\end{align}
Since on $\phi_3$, $s=r+it$ with $T\leq t\leq T+U$ we have 
\begin{align}
&\Delta_{\vphi_3} \arg(\Gamma(\alpha s + \beta))\\&= \alpha (T+U) \log \vt \beta - \alpha r  + \alpha (T+U)i \vt -
\alpha T \log \vt \beta -\alpha r + \alpha T i \vt - \alpha U   \\
& =   \alpha (T + U) \log (T+U) - \alpha T \log T + (\alpha \log \alpha - \alpha) U \\
& \qquad+ \alpha (T+U) \log\left \vt 1 + \frac{r-\beta/\alpha}{T+U}i \right\vt -
\alpha T \log \left\vt 1 + \frac{r-\beta/\alpha}{T}i \right\vt + O(1). 
\end{align}  
Note that $\log \vt 1 + \epsilon i \vt = o(\epsilon)$ and $\beta/\alpha, r \ll N^{\gamma}=o(T+U)$. Hence we have 
\begin{align}
\Delta_{\vphi_3} \arg(\Gamma(\alpha s + \beta))=  \alpha (T + U) \log (T+U) - \alpha T \log T + (\alpha \log \alpha - \alpha) U + O(N^\gamma)
\end{align}
% where we have taken $\beta/\alpha, r \ll T$ and used $\log \vt 1 + \epsilon i \vt = o(\epsilon)$. 
Similarly, by requiring $\delta \ll T$,
\begin{align} \Delta_{\vphi_3} \arg(\Gamma(\alpha (\delta-s) + \beta)) 
= & \ \alpha (T + U) \log (T+U) - \alpha T \log T \\ &+ (\alpha \log \alpha - \alpha) U + O(N^\gamma). \end{align}
Henceforth, from \eqref{def:Del} we can conclude that 
\begin{align}
& \Delta_{\vphi_3} \arg(G(s))  \\
= \ & \Delta_{\vphi_3} \arg ( \lambda^{2s-\delta} ) + O(N^\gamma) + \\
&2 \sum_{i=1}^k 
\alpha_i (T + U) \log (T+U) - \alpha_i T \log T + (\alpha_i \log \alpha_i - \alpha_i) U +  \\
= \ & -2U \log \lambda + 2 A 
\left( (T+U) \log (T+U) - T \log T \right) + 2 B U + O(N^\gamma)\label{bdphi3}
\end{align}
where $A = \sum_{i=1}^k \alpha_i$ and $B = \sum_{i=1}^k \alpha_i \log \alpha_i - \alpha_i$. 

\subsection{On the horizontal edges} \label{ssec:horizontal-edges}
We aim to upper bound the change in arguments on $\vphi_2$ by $O_a(N^\gamma \log T)$. 
From Theorem \ref{Strip} we are guaranteed that $\vt \zeta_N(-r+iT) \vt \geq 1 + \vt a \vt$,
so in particular $\zeta_N(-r+iT) - a \neq 0$.
Let $\arg(\zeta_N(-r+iT) - a) = \alpha$ and
define
\begin{align}  F_N^\dagger(s) := \sum_{i=1}^N \frac{\overline{a_n}}{\lambda_n^s} \ta 
\zeta_N^\dagger(s) := F_N^\dagger(s) + G(s) F_N^\dagger(\delta-s). \end{align}
Let us also define
\begin{align}g_N(s) := \zeta_N(s+iT) e^{-i\alpha}+ \zeta_N^\dagger(s-iT) e^{i\alpha} - 2 \proj_\alpha a, \end{align}
where $\proj_\alpha z$ is given in Lemma \ref{projlemma}. Since for any real number $s$, 
\begin{align} \zeta_N^\dagger(s-iT) = \overline{\zeta_N(s+iT)} \end{align}
then from \eqref{hNdef}
\begin{align}
g_N(s) = 2 \proj_\alpha h_N(s+iT)=2\Re\left(h_N(s+iT)e^{-i\alpha}\right)
\end{align}
for any $s\in \mathbb{R}$.
Thus, we have extended the projection of $h_N$ on $\alpha$ to a holomorphic function $g_N(s)$.
An upper bound on 
$\Delta_{\vphi_2} \arg (h_N(s))$ which is an upper bound on $\Delta_{\vphi_2} \arg (h_N(s)e^{-i\alpha})$ would by $\pi$ times one plus the number of zeros of $g_N(s)$ on $-r \leq \sigma \leq r$ and $t = T$.  
We apply Jensen's formula to $g_N(s)$ on a circle centered at $-r$
with radius $4r$. If $n(u)$ denotes the number of zeros of $g_N$ in the disk $D(-r,u)$, then
\begin{align} 
\frac{1}{2\pi} \int_0^{2\pi} \log \frac{ \vt g_N(-r + 4r e^{i \theta}) \vt}{\vt g_N(-r) \vt} \diff \theta
= \int_0^{4r} \frac{n(u)}{u} \diff u \geq \int_{2r}^{4r} \frac{n(u)}{u} \diff u \geq n(2r) \log 2. \label{jensenfor}
\end{align}
Note that all zeros of $g_N(s)$ on $\vphi_2$ are counted within $n(2r)$. 
%\textcolor{red}{
Additionally, our choice of $T_a$ allows us to apply
our estimates on $F_N(s)$ and $G(s)$. 
We are guaranteed that 
\begin{align}
	\vt g_N(-r) \vt \geq 2.\label{lbdgn}
\end{align}
%}
Applying lemmas \ref{Chi Bound} and \ref{Fn bound} 
on the boundary of $D(-r,4r)$ we have
\begin{align}
\vt g_N(s) \vt &\ll N^{5r \mu + \nu} \left( 1 + \lambda^{10r - \delta} 
\prod_{i=1}^k \left( \alpha_i (T+4r) \right)^{\alpha_i (10r - \delta)} \right).
\end{align}
Therefore
\begin{align}
\log \vt g_N(s) \vt \leq & (5r \mu + \nu) \log N + (10r - \delta) \log \lambda \\
&+ \sum_{i=1}^k \alpha_i (10r - \delta) \log( \alpha_i (T+4r)) + C
\end{align}
for some constant $C$. Since $r =\sigma_a N^\gamma$ and $N^\gamma \ll T$, we conclude that 
\begin{align}
\log \vt g_N(s) \vt \ll N^\gamma \log T.\label{ubdgn}
\end{align} 
Combining \eqref{jensenfor}, \eqref{lbdgn} and\eqref{ubdgn} we have 
\begin{align} \Delta_{\vphi_2} \arg(h_N(s)) \leq \pi(n(2r) + 1) &\ll \max_{0\leq \theta\leq 2\pi} \log \vt g(-r+4re^{i\theta}) \vt\ll N^\gamma \log T.\label{bdphi2}
\end{align}
%so it follows that $\Delta_{\vphi_2} \arg(h_N(s)) = O_a(N^\gamma \log T)$. 
By similar arguments, one has
\begin{align}
\Delta_{\vphi_4}\arg(h_N(s))\ll N^\gamma \log (T+U)).\label{bdphi4}
\end{align}

Now putting together equations \eqref{bdphi1}, \eqref{bdphi3}, \eqref{bdphi2} and \eqref{bdphi4} in \eqref{arghn}, we obtain that 
\begin{align} N_a(T,U) =&  \frac{A}{\pi}  \left( (T+U) \log (T+U) - T \log T \right) \\ &+ \frac{B - \log \lambda + \Psi}{\pi} U 
+ O_a(N^\gamma (T+U)), \end{align}
where $2 \Psi = \log \lambda_2$ if $a_1 = 0, a \neq 0$, $\Psi = -\log \lambda_2$ if $a = a_1 \neq 0$,
and $0$ otherwise. 
This concludes the proof. 

\section{Proof of Theorem \ref{Cluster}}

%      \textbf{Setup.}
% Like that of the previous proof, let $F^*(s) = F(s)-a$ and $\zeta^*_N(s) = F^*(s) + \chi(s) F^*(\delta-s)$. 
%Let $\sigma_a,\gamma > 0$ be such that all $a$-values of $\zeta_N$ and $\zeta_N^*$
% with sufficiently large $t$ have $\vt \sigma \vt \leq \sigma_a N^\gamma$. 
%Also, Pick $T_a$ in the same manner as the previous proof.

Let $N_1$ be the number of $a$-values for $\zeta_N$ with $\sigma > \hf[\delta] + \epsilon$,
and $N_2$ be the number of $a$-values for $\zeta_N$ with $\sigma < \hf[\delta] - \epsilon$ inside the heights $t=T$ and $t=T+U$. %We are interested in $N_a(T,U) - N_{a,\epsilon}(T,U) = N_1 + N_2$. 
Recall the equation \eqref{Reflection},
\begin{align} 
h_N(s)= \zeta_N(s) - a = G(s) (\zeta^*_N(\delta-s) + a). 
\end{align} 
Since $G(s)$ is non-zero for all $s\in \mathbb{C}$ with $\Im(s)\neq 0$, then $N_2$ is the same as the number of $-a$-values for $\zeta_N^*$ with $\sigma > \hf[\delta]+\epsilon$.
Hence, we will estimate $N_1$ first and an estimation for $N_2$ can be found similarly. 

Let $\sigma_r = \sigma_a N^\gamma$. We apply Littlewood's lemma to $h_N(s) = \zeta_N(s) - a$ on the rectangle $R$ bounded by the lines $\sigma = \hf[\delta]$, $\sigma= \sigma_r$, $t = T$ and $t=T+U$. 
Let $\rho$ denotes a zero of $h_N(s)$, then
\begin{align}
2 \pi \sum_{\rho \in R} \left(\Re(\rho) - \hf[\delta]\right) &= \int_{T}^{T+U} \log \vt \zeta_N(\hf[\delta] + it) - a \vt 
- \log \vt \zeta_N(\sigma_r+it) - a \vt \diff t \\ &\quad+
\int_{\hf[\delta]}^{\sigma_r} \arg (\zeta_N(\sigma+iT+iU)-a) - \arg(\zeta_N(\sigma+iT)-a) \diff \sigma. 
\end{align}
First we bound the first integral. Note that  $\left\vt \Delta\left(\hf[\delta] + it\right) \right\vt = 1$.
Therefore for some $\nu > 0$ and from Lemma \ref{Fn bound} we have $\vt \zeta_N(\hf[\delta] + it) \vt \ll N^\nu$ and hence 
\begin{align}  \int_T^{T+U} \log \vt \zeta_N(\hf[\delta]+it)-a \vt \diff t
\leq \nu U \log N+ O_a(U). \end{align}
For the second term in the first integral, if $a \neq a_1$, then from Theorem \ref{Strip} we know that
$ \vt \zeta_N(\sigma_r+it)-a \vt > \vt a - a_1 \vt/2 $, which gives  
\begin{align} \int_T^{T+U} 
\log \vt \zeta_N(\sigma_r+it)-a\vt \diff t \geq O_a(U). \end{align}
Instead, if $a = a_1$, then Theorem \ref{Strip} gives 
$\vt \zeta_N(\sigma_r+it)-a \vt > \lambda_2^{-\sigma_r} \vt a_2 \vt/2$.  Hence
\begin{align} \int_T^{T+U} 
\log \vt \zeta_N(\sigma_r+it)-a\vt \diff t > \left( -\sigma_r \log \lambda_2 + \log \vt a_2 \vt/2 \right) U = O(U).\label{1stintegral} \end{align}

Now we bound the second integral. Again from Theorem \ref{Strip} we find that $\arg(\zeta_N(\sigma_r+iT+iU)-a)$
and $\arg(\zeta_N(\sigma_r+iT))$ differs by at most $\pi$ along the right edge.
Moreover, from \eqref{bdphi2} and \eqref{bdphi4} we find that
$\arg(\zeta_N(\sigma+iT)-a)$ varies by up to $O_a(N^\gamma \log(T))$ and $\arg(\zeta_N(\sigma+iT+iU))$
varies by up to $O_a(N^\gamma \log(T+U))$ along the top and bottom edges.
Therefore 
\begin{align}
&\int_{\hf[\delta]}^{\sigma_r} \arg(\zeta_N(\sigma+iT+iU)-a)-\arg(\zeta_N(\sigma+iT)-a) \diff \sigma \\
\ll& (\sigma_r - \hf[\delta]) N^\gamma (\log(T+U) + \log(T)) \\ \ll_a& N^{2\gamma} \log(T+U). \label{2ndintegral} 
\end{align}
Combining \eqref{1stintegral} and \eqref{2ndintegral} we find 
\begin{align} 2 \pi \sum_{\rho \in R} \left(\Re(\rho) - \hf[\delta]\right) = O_a(N^{2\gamma} \log(T+U)) + O_a(U \log N). \end{align}
We can thus conclude that, $N_1$, the number of zeroes of $\zeta_N(s)$ with
$\Re(\rho) - \hf[\delta] > \epsilon$, is upper bounded by
\begin{align}N_1 = O_a(N^{2\gamma} \log(T+U) /\epsilon) + O_a(U \log N/\epsilon). \end{align}

Now observing from the definition that $N_a(T,U) -N_{a,\epsilon}(T+U)=N_1+N_2$. Since, the same computation works for $N_2$ when looking at zeros of $\zeta^*_N(s)+a$, so, we have
\begin{align}N_a(T,U) -N_{a,\epsilon}(T+U) = O_a(N^{2\gamma} \log(T+U) /\epsilon) + O_a(U \log N/\epsilon). \end{align}
This completes the proof of the theorem.

\section{Proof of Theorem \ref{Critical Zero}} From \eqref{def:zetN} we have 
\begin{align}
\zeta_N(s)=F_N(s)\left(1+G(s)\frac{F_N(\delta-s)}{F_N(s)}\right).
\end{align}
%Let us denote $G(s)=|\Delta|e^{i\Theta}$ and $F_N(s)=|F_N|e^{i\Phi}$. 
Since, on the critical line, $\vt G(\hf[\delta]+it) \vt = 1$, then one can write $G(\hf[\delta]+it)=e^{i\theta}$, where $ \vartheta=\vartheta(t)=\arg G(\hf[\delta]+it)$. Also, note that from \eqref{def:F} if $a_n$'s are real
then $F_N(\hf[\delta]-it) = \overline{F_N(\hf[\delta]+it)}$. Hence, if we denote $z =z(t)= F_N(\hf[\delta]+` it)$ and $\arg(z) = \phi=\phi(t)$ then for all real $a_n$ we have 
\begin{align} 
\zeta_N\left(\hf[\delta]+it\right) &= F_N\left(\hf[\delta]+it\right)\left(1+G\left(\hf[\delta]+it\right)\frac{F_N\left(\hf[\delta]-it\right)}{F_N\left(\hf[\delta]+it\right)}\right)\\
&=|z|e^{i\phi}(1+e^{i(\theta-2\phi)}).\label{znarg}
\end{align}
Henceforth, $\zeta_N(\hf[\delta]+it) = 0$ if and only if $|z(t)| = 0$,
or  $\vartheta(t) - 2\phi(t)$ is an odd-multiple of $\pi$. Also, if $F(\delta/2+ig)=0$ then we define
\begin{align}
\frac{F_N\left(\hf[\delta]-ig\right)}{F_N\left(\hf[\delta]+ig\right)}=\lim_{t\to g}\frac{F_N\left(\hf[\delta]-it\right)}{F_N\left(\hf[\delta]+it\right)}.
\end{align}
Next we  define the following three quantities. 
\begin{itemize}
\item[a)] $N_F^0$ be the cardinality of the set $S_F^0=\{\frac{\delta}{2}+it:F(\delta/2+it)=0, T\leq t\leq T+U\}$.
\item[b)]  $N_F^+$ be the cardinality of the set $S_F^+=\{\sigma+it:F(\sigma+it)=0, \sigma>\delta/2, T\leq t\leq T+U\}$.
\item[c)]  $N_Z^0$ be the cardinality of the set $S_Z^0$ consisting all $t$'s such $\zeta_N\left(\hf[\delta]+it\right)=0$  but $F_N\left(\hf[\delta]+it\right) \neq 0$ with $T < t < T+U$. In other words, $S_Z^0$ consists all $t$'s such that  $F_N\left(\hf[\delta]+it\right) \neq 0$ but $\theta-2\phi$ are odd-multiple of $\pi$ with $T < t < T+U$. 
\end{itemize}
Then, from the definition of $N_{a,0}$ we find that $N_{0,0}\geq N_F^0 + N_Z^0$. %, since $N_Z^0$ does not count by multiplicity.
Our goal to obtain a lower bound of $N_F^0 + N_Z^0$.

Consider the line segment $L_\epsilon$ between $\hf[\delta]+iT$ and
$\hf[\delta] + i(T + U)$ directed upward where about each element in $S_F^0$, a line segment of length $2\epsilon$
is replaced by a semicircle centered at that element of radius $\epsilon$ to the right of the zero.
We choose $\epsilon$ small enough so that no semicircles overlap,
no element of $S_F^+$ lies to the left of $L_\epsilon$, and that all occurrences of elements of $S_Z^0$
occur on the line segments of $L_\epsilon$. Also, consider the function $A(s)=\arg G(s)-2\arg F_N(s)$. Then $A\left(\hf[\delta]+it\right)=\vartheta(t) - 2\phi(t)$. Note that the set of all $t$ from the line segment part of $L_{\epsilon}$ such that the function $\vartheta(t) - 2\phi(t)$ intersect the line $t=(2k+1)\pi$ formed the set $S_Z^0$. Let $m(g)$ be the multiplicity of the zero $\hf[\delta]+ig$ of $F_N(s)$. Then we have 
\begin{align}
N_F^0=\sum_{T\leq g\leq T+U}m(g).\label{lbd1}
\end{align}
As limit $\epsilon \rightarrow 0$, the image of the semicircle part of $L_{\epsilon}$ surrounded the zero  $\hf[\delta]+iT$ gives a vertical line of the length $\pi m(g)$. Therefore if $\mathcal{N}$ is the maximum numbers of intersections of the lines $t=(2k+1)\pi$ and the vertical line segments of length $\pi m(g)$ coming from the image of the semicircle part of $L_{\epsilon}$ surrounded the zeros  $\hf[\delta]+ig$ then 
\begin{align}
N_Z^0\geq \lim_{\epsilon\to o}\frac{\left\lvert\Delta_{L_{\epsilon}}(A(s))\right\rvert}{2\pi}-\mathcal{N}+O(1).\label{lbd2}
\end{align}
Since the maximum number of intersections of lines $t=(2k+1)\pi$ and any vertical line segments of length $\pi m(g)$ is $m(g)$. Hence from \eqref{lbd1} and \eqref{lbd2} we find 
\begin{align}
N_Z^0+N_F^0\geq \lim_{\epsilon\to o}\frac{\left\lvert\Delta_{L_{\epsilon}}(\arg G(s)-2\arg F_N(s)\right\rvert}{2\pi}+O(1).\label{bb0}
\end{align}
%change of argument of $\phi - \hf[\vartheta]$
%on each semicircle becomes $\pi$ times the multiplicity  of the zero, 
%so $\hf[\vartheta] - \phi = k\pi$ occurs precisely the multiplicity of the zero times on the semicircle. 
%Henceforth, the number of times $\hf[\vartheta] - \phi = k \pi$ on all semicircles of $L_\epsilon$ is exactly $N_F^0$. However, $N_Z^0$ is the number of times $\hf[\vartheta] - \phi = k\pi$ on the line segments of $L_\epsilon$, so $N_Z^0 + N_F^0$ is total number of times $\hf[\vartheta] - \phi = k\pi$ on $L_\epsilon$. But this is lower bounded by the totalchange of $\hf[\vartheta] - \phi$,
% \begin{align} N_Z^0 + N_F^0 \geq \frac{1}{\pi} \vt \lim_{\epsilon \rightarrow 0} \Delta_{L_\epsilon} \left( \hf[\vartheta] - \phi \right) \vt,  \end{align}
% or in other words,
%\begin{align} N_Z^0 + N_F^0 \geq  \vt \lim_{\epsilon %\rightarrow 0}  \Delta_{L_\epsilon}   %\frac{\arg(\chi(s))}{2\pi} - \frac{\arg(F_N(s))}{\pi} \vt
%. \end{align}
%\textbf{Step 3.} Take $T_a$ so that Stirling's formula \ref{Stirling} holds on the critical line with $O(\vt s^{-1} \vt) < 1$.
Since $G(s)$ has no zeroes near $L_\epsilon$, to obtain a lower bound of $\Delta_{L_{\epsilon}}(\arg G(s))$, it suffices to evaluate the change of $G(s)$
over the straight line segments of $L_{\epsilon}$ joining $\hf[\delta] + iT$ and $\hf[\delta]+i(T+U)$. 
Using Lemma \ref{Stirling} and following the same steps of \eqref{bdphi3} , we have
\begin{align} 
\Delta_{L_{\epsilon}} \arg(G(s)) =  -2U \log \lambda +2 A   \left( (T+U) \log (T+U) - T \log T \right) + 2 B U + O(1). \label{bb1}
\end{align}
Next, we evaluate  $\Delta_{L_\epsilon} \arg(F_N(s))$ by using the argument principle. Let $\sigma_0$ be such that Lemma \ref{Fn bound} holds for $\sigma \geq \sigma_0$. That is, if $\sigma \geq \sigma_0$, then either $a_1 \neq 0$ and $\vt F_N(s) - a_1 \vt \leq a_1/2$, or $a_1 = 0$ and 
$\vt F_N(s) \lambda_2^s - a_2 \vt \leq a_2/2$. 
Consider the rectangular looking region with $L_\epsilon$ as the one vertical edge and the other vertical edge is the line segment from $\sigma_0 + iT$
to $\sigma_0 + iT + iU$, denoting it by $\phi_1$. Let us denote the horizontal edges by $\phi_3$ and $\phi_4$. It can be noted that all zeroes of $N_F^+$ lie in this region. %If again,$\vphi_1,\vphi_2,\vphi_4$ were the other sides of the rectangle with $\vphi_1$ being the right, 
Then
\begin{align} 2 \pi  N_F^+ = &- \Delta_{L_\epsilon} \arg(F_N(s)) + \Delta_{\vphi_1} \arg(F_N(s)) 
\\ &+ \Delta_{\vphi_3} \arg(F_N(s)) + \Delta_{\vphi_4} \arg(F_N(s)). \end{align}
Arguing similar to \eqref{bdphi1} we find, if $a_1 \neq 0$, then 
\begin{align}
\Delta_{\vphi_1} \arg(F_N(s)) = O(1).\label{b1}
\end{align}
Otherwise,
if $a_1 = 0$, then 
\begin{align}
\Delta_{\vphi_1} \arg(F_N(s)) = - U \log \lambda_2 + O(1).  \label{b2}
\end{align}
By our choice of $\sigma_0$, we have that $\vt F_N(\sigma_0 + iT) \vt \gg 1$. Let us denote $\alpha=\arg(F_N(\sigma_0+iT))$. Consider the function
\begin{align} 
g_N(s) = F_N(s+iT) e^{-i\alpha} + F_N(s-iT) e^{i\alpha}.
\end{align} 
Similar to \eqref{jensenfor}, apply Jensen's formula to the function $g_N(s)$ on the disk centered at $\sigma_0$ with radius $2 \sigma_0 + \delta$.
We have that $\vt g_N(\sigma_0) \vt \gg 1$, and 
$\vt g_N(s) \vt \ll N^{\mu (3 \sigma_0 + \delta) + \nu}$. Thus, we may conclude that 
\begin{align}
\Delta_{\vphi_3} \arg(F_N(s)) \ll \log N.\label{b3}
\end{align}
The same holds for $\vphi_4$, i.e., 
\begin{align}
\Delta_{\vphi_4} \arg(F_N(s)) \ll \log N.\label{b4}
\end{align}
Finally, by applying Jensen's formula to $F_N(s)$ directly about a circle encircling the entire rectangle with the same center, the radius of the circle is
of order $U$, so $\vt F_N(s) \vt \ll N^{\mu U + \nu}$,
and thus 
\begin{align}
N_F^+ = O(U \log N)
\label{b5}
\end{align}
Combining \eqref{b1}, \eqref{b2}, \eqref{b3}, \eqref{b4} and \eqref{b5} we conclude that 
\begin{align}
\Delta_{L_\epsilon} \arg(F_N(s))\ll U\log N.\label{bb2}
\end{align}
Substituting \eqref{bb1} and \eqref{bb2} in \eqref{bb0}, we obtain
\begin{align} N_Z^0 + N_F^0 \geq \frac{A}{\pi}   \left( (T+U) \log (T+U) - T \log T \right) + O(U \log N). \end{align} 
This completes the proof of the theorem. 

\section{Proof of Theorem \ref{Critical}}
Let $a\neq 0$. From \eqref{znarg} we have 
\begin{align} 
\zeta_N\left(\hf[\delta]+it\right)= |z|(e^{i\phi}+e^{i(\theta-\phi)}).
\end{align}
Hence, if $ \zeta_N(\hf[\delta]+it) \neq 0$ then \begin{align} 
\arg\left(\zeta_N\left(\hf[\delta]+it\right)\right) = \hf \arg\left(G \left(\hf[\delta]+it\right)\right). 
\end{align}
For a non-simple $a$-value of $\zeta_N(s)$ to occur on the line $\sigma=\delta/2$, it must be that 
\begin{align}
\pdv{}{t} \arg\left(\zeta_N\left(\hf[\delta]+it\right)\right) = 0.
\end{align} However,
for large enough $t$ and by the lemma \ref{Monotone} this will not happen, thereby the $a$-values with large enough $t$
on the critical line will all be simple. 

Now, on the critical line, $\vt G(\hf[\delta]+it) \vt = 1$. Let $\arg(a) = \alpha$. Since $a_n$'s are real, we can let $z(t) = F_N(\hf[\delta]+it)$ and $\overline{z}(t) = F_N(\hf[\delta]-it)$. 
We want to count the occurrences of $z(t) + G(\hf[\delta]+it)\overline{z}(t) = a$. By Lemma \ref{projlemma} and
by linearity of the projection we want 
\begin{align}
\proj_{\alpha} \left(z(t) + G(\hf[\delta]+it)\overline{z}(t)\right) = \vt a \vt \ta \proj_{\alpha+\pi/2} \left(z(t) + G(\hf[\delta]+it)\overline{z}(t)\right) = 0. \end{align}
But since $\vt z(t) \vt = \vt G(\hf[\delta]+it)\overline{z}(t) \vt$, by the Pythagorean Theorem it must be the case that 
$2 \proj_\alpha z = \vt a \vt$. 

Next, we upper bound the occurrences of $2 \proj_\alpha z(t) = \vt a \vt$. 
Once again, we extend the $2 \proj_\alpha F_N(\hf[\delta] + it)$ (excluding the constant term) 
to a holomorphic function
\begin{align} h(s) =  \sum_{n=2}^N a_n \lambda_n^{-\hf[\delta]-s}  e^{-i \alpha}
+ a_n \lambda_n^{-\hf[\delta]+s} e^{i\alpha} \end{align}
We want to upper bound the occurrences of $h(s) = \vt a \vt - 2 \proj_\alpha a_1$ on the imaginary axis.
Noting that $h$ is real on the imaginary axis, 
the number of times this can occur is at most one plus the number zeroes of the partial 
derivative of $h$ with respect to $t$ on the imaginary axis. 
Define
\begin{align} g(s) = \pdv{}{t} \phantom{} h(s) = 
\sum_{n=2}^N - i \log \lambda_n a_n \lambda_n^{-\hf[\delta]-\sigma} e^{- i \log \lambda_n t} e^{-i \alpha}
+ i \log \lambda_n a_n \lambda_n^{-\hf[\delta]+\sigma} e^{i \log \lambda_n t} e^{i\alpha}. \end{align}
We will argue using Jensen's formula that the number of zeroes of $g(s)$ on the imaginary axis with
$T < t < T+U$ is $O(U \log N)$ for sufficiently large $U$, and this is independent of the exact choice of $a$. 

To apply Jensen's formula, we need to lower bound $g(s)$ at the center.
Let $b_n = 2 a_n \lambda_n^{-\hf[\delta]} \log \lambda_n$. 
When $s$ is purely imaginary, 
\begin{align} g(it) = \sum_{n=2}^N b_n \sin \left( \log \lambda_n t + \phi_n \right) \end{align} 
for appropriate $\phi_n \in \R$.  We claim 
that if $U$ is sufficiently large, there always exists some $T < c < T+U$ for which 
$\vt g(ic) \vt \geq b_2/10$. To do this, consider the average value of 
\begin{align} f(t) = \sin(\log \lambda_2 t + \phi_2) 
\cdot g(it) \end{align}
for $T < t < T+U$. 
Using Lemma \ref{Trig} and assuming $U \gg \log \lambda_2$, we have 
\begin{align}
& \frac{1}{U} \int_T^{T+U} \sin(\log \lambda_2 t + \phi_2) \cdot g(it) \diff t \\
= & \ \frac{b_2}{U} \int_T^{T+U} \sin^2(\log \lambda_2 t + \phi_2) \diff t \\ 
&\qquad+ \sum_{n=3}^N
\frac{b_n}{U} \int_T^{T+U} \sin(\log \lambda_2 t + \phi_2) \sin(\log \lambda_n t + \phi_n) \diff t  \\
\geq & \ \frac{b_2}{3} - \sum_{n=3}^N \frac{b_n}{U} \frac{1}{\log \lambda_n / \lambda_2} .
\end{align}
Since $\vt a_n \vt$ and $\lambda_n$ are bounded by a polynomial in $N$, hence $b_n$ is also bounded by a polynomial in $N$.
Therefore, if 
\begin{align}
\sum_{n=3}^N b_n \leq N^\nu,
\end{align}
then by taking $U$ large enough so that
\begin{align} U \log \lambda_3/\lambda_2 \gg \frac{10}{b_2} N^\nu \end{align} 
for some appropriate $\nu > 0$ 
we have that the average value 
of $f(t) \geq b_2/10$. Thus, there must
exist some $t = c$ for which the inequality holds. But $\vt \sin(\log \lambda_2 c + \phi_2) \vt \leq 1$,
so we have $\vt g(ic) \vt \geq b_2/10$. 

The roots
of $g(s)$ we want to control on the imaginary line lie in $D(ic, U)$. 
We apply Jensen's formula to $g(s)$ centered at $ic$ with radius $2U$. We have that 
\begin{align} \frac{1}{2\pi} \int_0^{2\pi} \log \frac{\vt g(ic+2Ur^{i\theta}) \vt}{\vt g(ic) \vt}
\geq n(U) \log 2. \end{align}
However, $\vt g(ic) \vt$ is lower bounded by a fixed constant, and with lemma \ref{Fn bound}, 
$\vt g(s) \vt \ll N^{\mu \sigma + \nu}$ for some $\mu,\nu$. 
Thus, we have that 
\begin{align}\log \frac{\vt g(ic+2Ur^{i\theta}) \vt}{\vt g(ic) \vt} \ll (2U + \hf[\delta]) \log N. \end{align} 
So, it follows that $n(U) \leq O(U \log N)$, which concludes the proof.

\section{Acknowledgements}

The research was supported by the NSF-REU DMS-2150179 grant. The author thank
the University of North Carolina at Charlotte, in particular the mathematics and statistics department,
for their accommodations. 
%%%%%%%%%%%%%%%%

%%%%%%%%%%%%%%%%%

%\bibliography{ref}
%\bibliographystyle{abbrv}

\end{document}